\documentclass[12pt,reqno]{amsart}
\usepackage{a4wide}
\usepackage[T1]{fontenc}
\usepackage{lmodern}
\usepackage{amsmath,amsthm}
\usepackage{amsfonts}
\usepackage{epstopdf}
\usepackage {minitoc}
\usepackage {amscd}
\usepackage {amssymb}
\usepackage {graphics}
\usepackage{graphicx}
\usepackage{tikz}
\usepackage{enumerate}
\newcommand{\T}{\mathbb{T}}
\newcommand{\D}{\mathbb{D}}
\newcommand{\C}{\mathbb{C}}

\newcommand{\TT}{\mathbb T}

\newcommand{\bA}{{\bf{A}}}
\newcommand{\dist}{\mbox{dist} \,}

\newtheorem{theo}{Theorem}
\newtheorem{coro}{Corollary}
\newtheorem{lemma}{Lemma}
\newtheorem{rema}{Remarks}

\title[Cyclicity in Besov-Dirichlet spaces from the Corona Theorem]{Cyclicity in Besov-Dirichlet spaces from the Corona Theorem}

\author[Y. Egueh]{{Yabreb}  {Egueh}}
\address{Univ. Bordeaux\\ IMB, UMR 5251 F-33400 Talence\\ France \\ CNRS, IMB, UMR 5251, F-33400 Talence, France.}
\email{yabreb@math.u-bordeaux.fr}
\author[K . Kellay]{{Karim}  {Kellay}} 
\address{Univ. Bordeaux\\ IMB, UMR 5251 F-33400 Talence\\ France \\ CNRS, IMB, UMR 5251, F-33400 Talence, France.}
\email{kkellay@math.u-bordeaux.fr}
\author[M. Zarrabi]{{Mohamed} {Zarrabi}}
\address{Univ. Bordeaux\\ IMB, UMR 5251 F-33400 Talence\\ France \\ CNRS, IMB, UMR 5251, F-33400 Talence, France.}

\begin{document}
\begin{abstract}
Tolokonnikov's Corona Theorem is used to obtain two results on cyclicity in Besov-Dirichlet spaces. 
\end{abstract}
%\begin{altabstract}
%Nous utilisons le Théorème de la couronne  de Tolokonnikov pour obtenir deux résultats sur la cyclicité %dans les espaces de Besov-Dirichlet.
%\end{altabstract}

\maketitle
\vspace{-1em}
\hfill{\`A la m\'emoire de Mohamed Zarrabi}
 \section{Introduction}
 
 Let $X$ be a Banach space of analytic functions in the unit disc
$\D$ such that the shift operator $S : f(z) \to  z f(z)$ is a continuous map
of $X$ into itself. The cyclic vectors in $X$ are those functions $f$ such that
the polynomial multiples of $f$ are dense in $X$. Beurling in \cite{Beu} provided a complete characterization of cyclic vectors in the Hardy space; the cyclic vectors are precisely the outer functions. Cyclic vectors in the Dirichlet space were initially examined by Carleson in \cite{CA} and later by Brown and Shields in \cite{16}. In this paper, we focus on studying cyclic vectors in Besov-Dirichlet spaces. Specifically, motivated by the inquiries raised by Brown and Shields \cite[Question 3]{16} regarding cyclic vectors in a general Banach space $X$ of analytic functions :

\textit{Question:} If $f,g \in X$, if g is cyclic, and if $|f(z)| \geq |g(z)|$ for all $z\in \D$ then must $f$ be cyclic?

We extend some of Brown and Shields' results on cyclicity to Besov-Dirichlet spaces.  We now introduce some notations. For $p \geq 1$ and  $\alpha >-1$, the  Besov  space,  $\mathcal{D}^{p}_{\alpha} $ is the set of holomorphic functions on  $\mathbb{D }$ such that 
$$\mathcal{D}_{\alpha,p} (f)=\int_{\mathbb{D}}|f'(z)|^p { d A}_\alpha(z)<\infty, $$
where  $dA_\alpha(z)=(1+\alpha)(1-|z|^2)^\alpha dA(z) $
and  $dA(z)$  is the normalised Lebesgue measure on the disc. 
The Besov-Dirichlet space is equipped with the norm
 $$\|f\|_{\mathcal{D}^{p}_{\alpha}}^{p}=|f(0)|^p+\mathcal{D}_{\alpha,p} (f).$$

The Besov-Dirichlet space   $\mathcal{D}^{p}_{\alpha} $ is the set of holomorphic functions  $f$ on $\D$ whose derivative $f'$ is a function of the Bergman space $\mathcal{A}_{\alpha}^p =L^p(\D,dA_\alpha)\cap \textrm{Hol}(\D)$, where $\textrm{Hol}(\D)$ is the space of holomorphic functions on $\D$. Note that if $p = 2$ and $\alpha=1$, $\mathcal{D}_1^2$  is the Hardy space  $H^2$ and if $p = 2$ and  $\alpha= 0$, then  $\mathcal{D}_0^2$ is the classical Dirichlet space $\mathcal{D}$.\\

Denote by   $[f]_{\mathcal{D}^{p}_{\alpha}}$  the smallest $S$-invariant subspace containing  $f$,  the vector subspace generated by  $\{ z^nf,~ n \in \mathbb{N} \}$. We say that   $f\in \mathcal{D}^{p}_{\alpha}$ is cyclic in  $\mathcal{D}^{p}_{\alpha}$ if
$$[f]_{\mathcal{D}^{p}_{\alpha}}=\mathcal{D}^{p}_{\alpha}.$$
The function $f\in H^1$ is called outer function if it is of the form
$$f(z)=\exp\frac{1}{2\pi} \int_\TT \frac{\zeta+z}{\zeta-z}\log \varphi(\zeta){|d\zeta|}, \qquad |z|<1,$$
where $\varphi$ is nonnegative function in $L^1(\TT)$ such that $\log\varphi \in L^1(\TT)$. Note that $|f|=\varphi$ $a.e.$ on une unit circle $\TT=\partial\D$.

The problem of characterizing the cyclic vectors in the Dirichlet space $\mathcal{D}^2_0$ is much more difficultIn \cite{16},  Brown and Shields conjectured that a function $f$ in the Dirichlet space $\mathcal{D}$ is cyclic for the shift operator if and only if f is outer and its boundary zero set is of logarithmic capacity.  The characterization of cyclic vector of  $  \mathcal{D}_\alpha^p$  depends on the values of $p$ and  $\alpha$. More precisely  our investigation is limited to the case  $\alpha+1\leq p\leq \alpha +2$.  Indeed,  
   If $1<p < \alpha + 1$, then  $H^p $ is continuously embedded in $  \mathcal{D}_\alpha^p$ see \cite{34},  hence  every outer function $f\in H^p$ is cyclic for $  \mathcal{D}_\alpha^p =\mathcal{A}_{\alpha-p}^p $. On the other hand,  if  $p > \alpha + 2$, then $  \mathcal{D}_\alpha^p \subset\mathcal{A}(\mathbb{D})=\textrm{Hol}(\D)\cap \mathcal{C}(\overline{\D})$ becomes Banach algebra see \cite{34}, consequently the only cyclic outer function are the invertible functions,  and then any function that vanishes at least at one point is not cyclic in  $  \mathcal{D}_\alpha^p$. Let  $\mathcal{A}({\mathbb{D}})$ be the disc algebra. For $f\in\mathcal{A}({\mathbb{D}})$,  denote 
   $$\mathcal{Z}(f)=\{\zeta \in  \mathbb{T}: f(\zeta)=0\}$$ the zero set of $f$.  
   Recall that Brown and Shields conjectured that a function $f$ in the Dirichlet space $\mathcal{D}$ is cyclic for the shift operator if and only if f is outer and its boundary zero set is of logarithmic capacity.  Here, we will prove the following theorem.

\begin{theo} \label{239} Let $p>1$ such that  $\alpha+1\leq p\leq \alpha +2$.
Let $f \in \mathcal{D}^p_{\alpha} \cap\mathcal{A}({\mathbb{D}})$ be an outer function such that  $\mathcal{Z}(f)=\{1\}$, then  $f$ is cyclic in  $\mathcal{D}^p_{\alpha}.$
\end{theo}

The case of the classical Dirichlet space  $\mathcal{D}_0^2$ was discovered by Hedenmalm-Shields \cite{28} and generalized by Richter-Sundberg \cite{RS}. This result was shown \cite{34}  for $\alpha+1<p\leq \alpha +2$, the method used for the proof  is inspired by that of  Hedenmalm and Shields \cite{28}. Note that our result also includes the case where  $p=\alpha+1$.  Thanks to  \cite[Theorem 3]{28}, Theorem \ref{239}  remains true if $\mathcal{Z}(f)=\{1\}$ to a point is replaced by 
  $\mathcal{Z}(f)$ is countable.   Our second main result is

\begin{theo}\label{bbss1}
Let  $p>1$ such that  $1+\alpha\leq p\leq \alpha +2$.  Let  $f,g \in  \mathcal{D}^{p}_{\alpha}   \cap \mathcal{A}({\mathbb{D}})$ such that 
\begin{equation}\label{majorationfg}
|g(z)|\leq |f(z)| ,\qquad z\in \mathbb{D}.
\end{equation}
If $g$ is cyclic in $\mathcal{D}_{\alpha}^{p}  $ then $f$ is cyclic in  $\mathcal{D}_{\alpha}^{p} .$
\end{theo}
This result generalizes that of Brown and Shields \cite[Theorem 1]{16}, then Aleman \cite[Corollary 3.3]{Al} for $\mathcal{D}^2_\alpha$ spaces. \\

The proof of the two theorems is based on the Tolokonnikov Corona Theorem \cite{T}.  The idea of using Corona's theorem in this context goes back to Roberts for the Bergman space \cite{Ro}, see also \cite{AT,BEK,BEK0, EKS}. For some results related to cyclic vectors, see \cite{6,7,23,24,25,EKR2,RS2,RS} and the references therein.

\section{Proof of Theorem \ref{239} and Theorem \ref{bbss}}

We recall two results we will need for the proofs. 
The first is the  Corona Theorem of Tolokonnikov \cite{T}
\begin{theo}\label{cou} Let $1<p\leq \alpha+2$ and 
Let  $f_1 , f_2 \in \mathcal{D}_{\alpha}^{p} \cap\mathcal{A}({\mathbb{D}})$ such that  
 $$\sup_{z\in \mathbb{D}} \Big(|f_1(z)|+|f_2(z)|\Big)> \delta>0.$$
Then there exists $h_1 , h_2 \in \mathcal{D}_{\alpha}^{p}  \cap\mathcal{A}({\mathbb{D}})$ such that 
 $$\left\{\begin{array}{lll}
 f_1(z)h_1(z)+ f_2(z)h_2(z)=1, \qquad z\in \mathbb{D}\\
\\
 \| h_1\|_{\mathcal{D}_{\alpha}^{p}  \cap\mathcal{A}({\mathbb{D}})} \leq \delta^{-\bA} \quad \text{ and }\quad   \| h_2\|_{  \mathcal{D}_{\alpha}^{p} \cap\mathcal{A}({\mathbb{D}})} \leq \delta^{-\bA}
 \end{array}
 \right.
 $$
for some positive constant $\bA\geq4$  independent of  $p$ and  $\alpha$. 
 \end{theo}
 
 \begin{rema}
   If  $p>\alpha+2$,  then $\mathcal{D}^{p}_{\alpha}\subset \mathcal{A}(\mathbb{D})$.  If $p=2$ and $\alpha=0$, $\mathcal{D}^2_0=H^2$ and we therefore find the classical Carleson-Corona  Theorem \cite{Car} . In this case, the constant $\bA>2$ instead of $\bA\geq 4.$ If $\alpha=p-2$, Tolokonnikov \cite{T} showed that $\bA=4$. Nicolau in \cite{Nico} showed the Corona Theorem   but without giving the quantitative version, see also \cite{ARSW, CSW}.  
    \end{rema}

Let $T$ be a bounded linear operator actings on an infinite dimensional complex Banach space $X$. The spectrum of T is denoted by $\sigma(T)$. The following corollary is easily obtained by  Atzmon's Theorem  \cite{At} and Cauchy's inequalities.

\begin{coro}\label{atz} 

Let $T$ be an invertible operator on  Banach $X$ such that  $\sigma(T)=\{1\}$. Suppose that there exists $k\geq 0$  and  $c>0$  such that for $\varepsilon >0$, there exists  $ c_\varepsilon >0$
$$
\left\{ \begin{array}{ll}
 \|(T-zI)^{-1}\| \leq \displaystyle c_\varepsilon \exp\frac{\varepsilon}{1-|z|} &|z|<1,\\
 &\\
 \|(T-zI)^{-1}\| \leq\displaystyle \frac{c}{(|z|-1)^{k}} & |z|>1, 
 \end{array}
\right.
$$
then   $(I-T)^k=0$.
\end{coro}

\subsection{Proof of Theorem \ref{239}}

 
Let $\lambda \in \C$ and put
 $$\delta_\lambda:=\inf_{z\in \D}|\lambda-z|+|f(z)|.$$
Since $f$ is an outer function, by \cite{SS}
 $$\lim_{|z|\to 1-} (1-|z|)\log 1/|f(z)|=0.$$
For all $\varepsilon >0$, there is therefore  $ c_\varepsilon >0$  such that 
   \begin{equation}\label{fouter}
   |f(z)|\geq c_\varepsilon \exp\frac{-\varepsilon}{1-|z|},\qquad z\in \mathbb{D}.
   \end{equation}
Considering  $|\lambda|\neq 1$,  we distinguish two cases :\\ 

\begin{itemize}
    \item If $|z-\lambda|\geq  |1-|\lambda||/2$, then $\delta_\lambda\geq |1-|\lambda||/2$.\\

\item If $|z-\lambda|\leq  |1-|\lambda||/2$,  then 
$$
|1-|\lambda||/2\geq |z-\lambda|\geq |(1-|\lambda|)-(1-|z|)|
 \geq |1-|\lambda||-|1-|z||.$$
Thus, we get $1-|z|\geq  |1-|\lambda||/2$ and by \eqref{fouter}. We then have 
$$|f(z)| \geq c_\varepsilon \exp\frac{-\varepsilon}{|1-|\lambda||}.$$
\end{itemize}

Therefore, we finally get $$\delta_\lambda \geq c_\varepsilon \exp\frac{-\varepsilon}{|1-|\lambda||}.$$
According to the Theorem  \ref{cou}, There is $g,h\in \mathcal{D}_\alpha^p \cap\mathcal{A}({\mathbb{D}})$ such that 
$$\left\{
\begin{array}{lll}
(\lambda-z)g+fh=1\\
\\
\|g\|_{\mathcal{D}_\alpha^p \cap\mathcal{A}({\mathbb{D}})} \leq \delta_\lambda^{-\bA} \text{ and } \|h\|_{\mathcal{D}_\alpha^p \cap\mathcal{A}({\mathbb{D}})} \leq \delta_\lambda^{-\bA}
\end{array}
\right.
$$
for some constant $A\geq 4$.

Let $[f]_ {\mathcal{D}_\alpha^p \cap\mathcal{A}({\mathbb{D}})}$ be the ideal generated by  $f$ and let 
$$\pi : \mathcal{D}_\alpha^p \cap\mathcal{A}({\mathbb{D}}) \to  \mathcal{D}_\alpha^p \cap\mathcal{A}({\mathbb{D}}) /[f]_ {\mathcal{D}_\alpha^p \cap\mathcal{A}({\mathbb{D}})}$$ be the canonical surjection. We have 
$$(\lambda\pi(1)-\pi(z))^{-1}=\pi(g).$$

For $|\lambda|<1$, we have

\begin{eqnarray*} \|(\lambda\pi(1)-\pi(z))^{-1}\| &=&\|\pi(g)\|\\
&\leq &\|g\|_{\mathcal{D}_\alpha^p \cap\mathcal{A}({\mathbb{D}})} \\
&\leq& c_\varepsilon \exp\frac{\varepsilon}{1-|\lambda|}.
\end{eqnarray*}

For $|\lambda|>1$, we have 
\begin{eqnarray*}  \|(\lambda\pi(1)-\pi(z))^{-1}\|&\leq & \|(\lambda-z)^{-1}\|_{\mathcal{D}_\alpha^p \cap\mathcal{A}({\mathbb{D}})}\\
&=&\frac{1}{|\lambda|-1}+\frac{1}{|\lambda|}+\Big(\int_\D\frac{dA_\alpha(z)}{|\lambda-z|^{2p}}\Big)^{1/p}\\
&\leq&\frac{2}{|\lambda|-1}+ \frac{1}{(|\lambda|-1)^{p}}.
\end{eqnarray*}
The spectrum of $\pi$,  $\sigma(\pi)=\{1\}$, by Corollary \ref{atz},  $(\pi(1) -\pi(\alpha))^{[p]+1}=0$, and we get  $(1-z)^{[p]+1}\in [f]_ {\mathcal{D}_\alpha^p \cap\mathcal{A}({\mathbb{D}})}$. Since  $(1-z)^{[p]+1}$ is cyclic in $\mathcal{D}_\alpha^p$,   $f$ is also cyclic in  $\mathcal{D}_\alpha^p$ and  the proof is  complete.

\subsection{Proof of Theorem \ref{bbss1} } 

The proof of the Theorem \ref{bbss1} is deduced from the following two results.
\begin{lemma}
Let  $p>1$ such that  $1+\alpha\leq p\leq \alpha +2$.  Let  $f,g \in  \mathcal{D}^{p}_{\alpha}   \cap \mathcal{A}({\mathbb{D}})$, if $fg$ is cyclic, then both $f$ and $g$ are cyclic.
\end{lemma}
\begin{proof}
It suffices to show that $g$ is cyclic. Let ${\sigma_n}(f)$ denote the Fej\'er means of the partial sums of the power series for $f$. 
Since  the ${\sigma_n}(f)$ converges to $f$ in $\mathcal{D}^{p}_{\alpha}$, ${\sigma_n}(f)g$ converge pointwise to $fg$ in $\D$ and $$
    \|(\sigma_n(f)g-fg)'\|^p_{\mathcal{A}^p_\alpha}\leq  \|(\sigma_n(f)-f)\|_\infty \|g'\|^p_{\mathcal{A}^p_\alpha}+ \|\sigma_n(f)-f\|^p_{\mathcal{D}^p_\alpha}\|g\|_\infty,
$$
we obtain
 ${\sigma_n}g$ converge to $fg$ in $\mathcal{D}^{p}_{\alpha}$, which completes the proof.
\end{proof}

The constant $N$ in the following theorem  is related to that of the Corona Theorem.  If $\alpha=p-2$, we have $N(p-2,p)=5$. 

\begin{theo}\label{bbss}
    Let $p>1$ be such that $\alpha+1\leq p\leq \alpha+2$, there exists $N=N(\alpha,p)$ which depends only on $\alpha$ and $p$ such that if $f,g\in \mathcal{D}^{p}_{\alpha} \cap \mathcal{A}({\mathbb {D}})$ with 
$$
|g(z)|\leq |f(z)|,\qquad z\in \mathbb{D},
$$
then 
$[g^N]_{\mathcal{D}_{\alpha}^{p}}\subset [f]_{\mathcal{D}_{\alpha}^{p}}$ .
\end{theo}

 \begin{proof}
Let $\lambda\in \mathbb{C}$ and set 
$$\inf_{z\in \mathbb{D}} \Big\{|1-\lambda g(z)|+|f(z)| \Big \}=\delta_\lambda.$$ 
Considering  $\lambda \neq 0$, we have 
 
 \begin{itemize}
\item  If  $\displaystyle |g(z)| \leq \frac{1}{2 |\lambda|}$, then $ |1-\lambda g(z)| \geq 1- |\lambda|| g(z)| \geq \frac{1}{2}$.\\
 
 \item If $\displaystyle |g(z)| \geq \frac{1}{2 |\lambda|}$ then
$\displaystyle
|f(z)| \geq \frac{1}{2|\lambda|} $
\end{itemize}
From this, follows   $$\delta_{\lambda}  \geq  \frac{1}{2|\lambda|}  .$$
According to the  Theorem \ref{cou}, there are $F_\lambda, G_\lambda  \in  \mathcal{D}_{\alpha}^{p} \cap\mathcal{A}({\mathbb{D}})$ such that 
$$\left\{\begin{array}{lll}
 (1-\lambda g)G_\lambda+f F_\lambda=1,\\
\\
 \|F_\lambda\|_{\mathcal{D}_{\alpha}^{p} \cap\mathcal{A}({\mathbb{D}})} \leq \delta_\lambda^{-\bA} \quad \text{ and }\quad   \|G_\lambda\|_{\mathcal{D}_{\alpha}^{p} \cap\mathcal{A}({\mathbb{D}})} \leq \delta_\lambda^{-\bA}
 \end{array}
 \right.
 $$
 for some constant $\bA>4$ .

As before, we consider the canonical  surjection 
$$ \pi : \mathcal{D}_{\alpha}^{p} \cap\mathcal{A}({\mathbb{D}})  \to  \left( \mathcal{D}_{\alpha}^{p} \cap\mathcal{A}({\mathbb{D}})\right) / [f]_{\mathcal{D}_{\alpha}^{p} \cap\mathcal{A}({\mathbb{D}})} .$$ We have   $$ (\pi (1-\lambda g) )^{-1}=\pi(G_\lambda)$$ and 
\begin{align*}
 \|( \pi(1-\lambda g))^{-1} \|_{\mathcal{D}_{\alpha}^{p} \cap \mathcal{A}({\mathbb{D}})/[f]_{\mathcal{D}_{\alpha}^{p}  \cap\mathcal{A}({\mathbb{D}}) }}
 & =\|\pi(G_\lambda)  \|_{ \mathcal{D}_{\alpha}^{p} \cap\mathcal{A}({\mathbb{D}})/[f]_{\mathcal{D}_{\alpha}^{p}  \cap\mathcal{A}({\mathbb{D}})}  } \\
 & \leq \|G_\lambda  \|_{ \mathcal{D}_{\alpha}^{p} \cap\mathcal{A}({\mathbb{D}})  }\\
 &\leq   2^\bA|\lambda|^\bA .
\end{align*}
 By Liouville Theorem,  $\pi(1-\lambda g)^{-1} $ is polynomial of degree  at most $[\bA]$.    Since  $|\lambda g| < 1$, 
$ \pi(1-\lambda g)^{-1} 
=\sum_{  n \geq 0} \lambda^{n} \pi^n ( g) $. We obtain $\pi^{[A]+1}(g)=0$ which means that $g^{[\bA]+1} \in  [f]_{\mathcal{D}_{\alpha}^{p}  \cap\mathcal{A}({\mathbb{D}})}$ and hence  $[g^{[\bA]+1}]_{\mathcal{D}^{p}_{\alpha} }  \subset [f]_{\mathcal{D}^{p}_{\alpha} } $.

\end{proof}

\section{ Refinement of the Theorem  \ref{bbss1} }

We can improve the estimate \eqref{majorationfg} in Theorem \ref{bbss}, and thus obtain a possible more same conclusion. The improved estimate we are looking for is given by the following result.

\begin{theo}\label{rafinement} Let $p>1$ such that  $\alpha+1\leq p\leq \alpha +2$. Let $f,g\in \mathcal{D}_{\alpha}^{p} \cap \mathcal{A}({\mathbb{D}})$. Suppose that ${\rm Re}(g) \geq 0$ and there exists $\gamma >1$  such that. 
\begin{equation}\label{eqd1}
|g(z)|\leq \Big(\log\frac{\|f\|_{\mathcal{D}^p_\alpha \cap \mathcal{A}(\mathbb{D})}}{|f(z)|}\Big)^{-\gamma},\qquad z\in \mathbb{D}
\end{equation}
then $ [g] _{ \mathcal{D}^p_\alpha  } \subset  [f] _{ \mathcal{D}^p_\alpha  } .$  
\end{theo}

\begin{proof}

We assume that $\|f\|_{\mathcal{D}^p_\alpha \cap \mathcal{A}(\mathbb{D})}=1.$
Let $\lambda\in \mathbb{C}$, we set
$$\inf_{z\in \mathbb{D}} \Big\{|1-\lambda g(z)|+|f(z)| \Big \}=\delta_\lambda.$$
Considering $\lambda\neq 0$, we distinguish two cases:\\

\begin{itemize}

 \item If  $\displaystyle |g(z)| \leq \frac{1}{2 |\lambda|}$  then  $ |1-\lambda g(z)| \geq 1- |\lambda|| g(z)| \geq \frac{1}{2}$. \\
 \item If  $\displaystyle|g(z)| \geq \frac{1}{2 |\lambda|} $, then by \eqref{eqd1}
$$|f(z)| \geq e^{-\left(2|\lambda| \right)^{\frac{1}{\gamma}}}.$$
\end{itemize}
From this follows 
 $$\delta_{\lambda}  \geq e^{-(2|\lambda|)^{\frac{1}{\gamma}}}  .$$
By Theorem \ref{cou}, there exists  $F_\lambda, \; G_\lambda  \in  \mathcal{D}_{\alpha}^{p} \cap\mathcal{A}({\mathbb{D}})$ such that 
$$\left\{\begin{array}{lll}
 (1-\lambda g)G_\lambda+f F_\lambda=1,\\
\\
 \|F_\lambda\|_{\mathcal{D}_{\alpha}^{p} \cap\mathcal{A}({\mathbb{D}})} \leq \delta^{-\bA} \quad \text{ and }\quad   \|G_\lambda\|_{\mathcal{D}_{\alpha}^{p} \cap\mathcal{A}({\mathbb{D}})} \leq \delta^{-\bA}
 \end{array}
 \right.
 $$
for some constant $\bA\geq 4$.

Let $\pi$ be the canonical surjection 
$$\pi   : \mathcal{D}_{\alpha}^{p} \cap\mathcal{A}({\mathbb{D}}) \to \left( \mathcal{D}_{\alpha}^{p} \cap\mathcal{A}({\mathbb{D}})\right) / [f]_{\mathcal{D}_{\alpha}^{p} \cap\mathcal{A}({\mathbb{D}})}.
 $$
  We have 
 $$\pi (1-\lambda g)^{-1}=\pi(G_\lambda)$$ 
 and 
\begin{align*}
 \|  \pi(1-\lambda g)^{-1} \|_{\mathcal{D}_{\alpha}^{p} \cap \mathcal{A}({\mathbb{D}})/
 [f]_{\mathcal{D}_{\alpha}^{p}  \cap\mathcal{A}({\mathbb{D}})    }}
 & =\|\pi(G_\lambda)  \|_{ \mathcal{D}_{\alpha}^{p} \cap\mathcal{A}({\mathbb{D}})/[f]_{\mathcal{D}_{\alpha}^{p} }  } \\
 & \leq \|G_\lambda  \|_{ \mathcal{D}_{\alpha}^{p} \cap\mathcal{A}({\mathbb{D}})  }\\
 & \leq \frac{1}{\delta_{\lambda}^{\bA}} \leq e^{\bA (2|\lambda|)^{\frac{1}{\gamma}}}.
\end{align*}
Let $\ell \in ( { \mathcal{D}_{\alpha}^{p}\cap\mathcal{A}({\mathbb{D}})/[f]_{\mathcal{D}_{\alpha}^{p}\cap\mathcal{A}({\mathbb{D}})  }})^{*}$  with  norm $\|\ell\|=1$ and define $\varphi$ by $$
\varphi(\lambda)=   \langle (\pi(1-\lambda g))^{-1}, \ell \rangle.$$
 The function $\varphi $ is analytic on  $\mathbb{C}$ and 
\begin{align}\label{inPL}
| \varphi (\lambda) |
&  \leq e^{c|\lambda|^{\frac{1}{\gamma}}} 
\end{align}
where $c= 2^\frac{1}{\gamma}\bA$.  Since  $\gamma>1$, there exists $\theta  _\gamma$ such that  $\displaystyle \frac{\pi}{2} (2-\gamma) <\theta_\gamma <\frac{\pi}{2} \gamma$. We suppose that  $\theta  _\gamma<\pi$. 
Consider now te sector $\mathcal{S}_{\theta  _\gamma}= \{\lambda \in \mathbb{C} \text{ :  }|\arg \lambda| <\theta  _\gamma\}. $

{

\begin{figure}[!ht]
\centering
\setlength{\unitlength}{3cm}
\begin{picture}(1,1)
\put(0,0){\line(0,1){1}}
\put(0,0){\line(1,0){1}}
\put(0,0){\line(0,-1){1}}
\put(0,0){\line(-1,0){1}}
\put(-0.4,0.9){$\partial S_{\theta  _\gamma}$}
\put(-0.4,-0.9){$\partial S_{\theta  _\gamma}$}
\put(0,0){\line(-1,2){.5}}
\put(0,0){\line(-1,2){.5}}
\put(0,0){\line(-1,-2){.5}}
\put(0,0){\line(-1,-2){.5}}
\put(0.5,0.5){$S_{\theta  _\gamma}$}
\end{picture}
\vspace{7em}
\small \caption{le secteur $S_{\theta  _\gamma}$}\label{fig}
\end{figure}
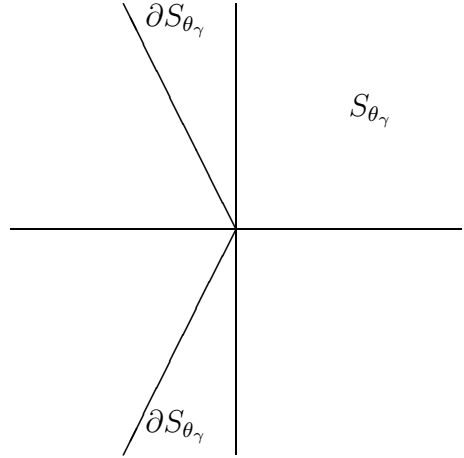
}

Let  $\lambda \in \partial \mathcal{S}_{\theta  _\gamma} $, since  $\pi/2<\theta  _\gamma<\pi$,   $\text{\rm Re}(\lambda)\leq 0 $ and  $\text{\rm Re}\displaystyle (\frac{1}{\lambda}-g(z))\leq 0$. We obtain 
\begin{align*}
|1-\lambda g(z)|
& = |\lambda||\frac{1}{\lambda}- g(z)| \\
& \geq |\lambda||{\rm Re} (\frac{1}{\lambda}- g(z))| \\
&\geq |\lambda|\frac{|{\rm Re} (\lambda)|}{|\lambda|^{2}} =\frac{|{\rm Re} (\lambda)|}{|\lambda|}.
\end{align*}

Moreover, $\lambda \in \partial \mathcal{S}_{\theta  _\gamma} $, hence 
$$  {\rm Re}(\lambda)= |\lambda| \cos\theta  _\gamma.$$
If we set $C_\gamma = |\cos\theta  _\gamma|^{-1}\neq 0$, we get
  $$\frac{1}{|1-\lambda g(z)|} \leq \frac{|\lambda|}{|{\rm Re} (\lambda)|}=C_\gamma.$$ 
Then $\varphi$ is analytic on $\mathcal{S}_{\theta  _\gamma}$, continuous on  $\overline{\mathcal{S}_{\theta  _\gamma}}$   and satisfies 
$$\left\{\begin{array}{ccccl}
 |\varphi(\lambda)| &\leq &e^{ c|\lambda|^{\frac{1}{\gamma}}} & \text{ for  } & \lambda \in\mathcal{S}_{\theta  _\gamma}  \\
&&&\\
|\varphi(\lambda)| &\leq &C_\gamma & \text{ for }  &\lambda \in \partial \mathcal{S}_{\theta  _\gamma} .
 \end{array}
 \right.
 $$

Since  $\displaystyle \frac{1}{\gamma}<\frac{\pi}{2\theta  _\gamma}$ with $\displaystyle \theta  _\gamma < \frac{\pi}{2}\gamma$,  by the Phragm\'en–Lindel\"of principle for a sector $\mathcal{S}_{\theta  _\gamma}$, we have 
  $$
|\varphi(\lambda)| \leq C_\gamma,  \qquad   \lambda \in  \mathcal{S}_{\theta  _\gamma}. $$
The function $\varphi$ is  an entire function  and  satisfies \eqref{inPL} on  $\mathbb{C}$.  Again using the Phragmén–Lindelöf principle for a sector
$$\mathcal{S}=\mathbb{C}\setminus  \mathcal{S}_{\theta  _\gamma}=\{\lambda\in \mathbb{C}\text{ : } \theta  _\gamma<\arg(\lambda)<2\pi-\theta  _\gamma\}.$$
Since  $2\pi-2\theta  _\gamma$ we get 
$$\left\{\begin{array}{ccccl}
 |\varphi(\lambda)| &\leq &e^{c|\lambda|^{\frac{1}{\gamma}}} & \text{ on }  & \lambda \in\mathcal{S}  \\
&&&\\
|\varphi(\lambda)| &\leq &C_\gamma & \text{sur}  &\lambda \in \partial \mathcal{S} 
 \end{array}
 \right.
 $$ 
Since  $\theta _\gamma>\displaystyle \frac{\pi}{2} (2-\gamma) $,   $\displaystyle \frac{1}{\gamma}<\frac{\pi}{2\pi-2\theta_\gamma}$ and 
$$| \varphi (\lambda) | \leq   C_\gamma \qquad \lambda \in  \mathcal{S}. $$
Then  $\varphi$  is bounded on  $\mathbb{C}$,  By  Liouville Theorem, $\varphi$ is a constant function
$$\varphi(\lambda)=\varphi(0)=\langle  \pi^{-1}(1),\ell \rangle,\qquad \lambda\in \mathbb{C}.$$
Thus, $\pi^{-1}(1-\lambda g)= \pi^{-1}(1)=\pi(1)$. For $|\lambda g| < 1$, we have 
$
\pi(1)
 =\pi^{-1}(1-\lambda g)
=\sum_{  n \geq 0} \lambda^{n} \pi^{n}(g)$.  Consequently
 $\pi(g)=0$ and $ g \in  [f]_{\mathcal{D}_{\alpha}^{p}\cap\mathcal{A}(\mathbb{D}) }$, hence $[g]_{\mathcal{D}^{p}_{\alpha} }  \subset [f]_{\mathcal{D}^{p}_{\alpha} }  $.

\end{proof}

\begin{rema} We will construct two functions $f$ and $g$ that satisfy the condition of the Theorem \ref{rafinement}, this answers a question of  Sasha Borichev.  A closed set $E$ of the unit circle is said to be 
 $K$-set  (after Kotochigov), if there exists a positive constant $c_E$ such that  for any arc $I\subset \T$ 
$$\sup_{\zeta\in I} \dist(\zeta, E) \geq c_E| I|$$
where $I$ denotes the length of $I$. K-sets arise as the interpolation sets for  H\"older classes, as example the generalized cantor set \cite{24,25},  we refer to \cite{B,D} for more details.  Such a set fulfils the following condition 
$$\frac{1}{|I|}\int_I \frac{|\zeta|}{\dist(\zeta,E)^\alpha}\le |I|^{-\sigma}$$
for 
$$\sigma < \Big(\log\big(\frac{1}{1-c_E}\big)\Big)\Big/ \Big(\log\big(\frac{2}{1-c_E}\big)\Big).$$
In particular, $E$  has measure zero and $\log \dist(\zeta, E) \in L^1(\T)$. Let $p>1$ such that  $\alpha+1\leq p\leq \alpha +2$. Let us now  consider the outer function  
$$|g(\zeta)|=\dist(\zeta,E)^\beta,\qquad \zeta\in \T.$$ 
Since $E$ is $K$-set,  by \cite{B}, $\Re g(z)>0$,  
$$\Re g(z)\asymp |g(z)|\asymp \dist(z,E)^\beta \quad \text{ and }  
\quad |g'(z)|\asymp \dist(z,E)^{\beta-1},\qquad z\in \D.$$
If $1/(2+\alpha)<\beta <1$, then 
$$\mathcal{D}_{\alpha,p}(g)\asymp \int_\D\frac{dA_\alpha(z)}{\dist(z,E)^{p(1-\beta)}}\lesssim \int_{0}^{1}\frac{dr}{(1-r)^{(\alpha+2)(1-\beta)-\alpha}}<\infty.$$
so   $g\in \mathcal{A}(\D)\cap  \mathcal{D}_{\alpha}^{p}$. Now let  $1/\gamma=\kappa$
 $$f(z)=\exp(-1/g^\kappa(z)), \qquad z\in \D.$$
We have $\displaystyle f'(z)=\kappa\frac{g'(z)}{g(z)^{\kappa+1}}\exp(-1/g^\kappa(z))$. Thus $|f'(z)|\leq |g'(z)|$ and $f\in \mathcal{D}_{\alpha}^{p}$. \\

Let us conclude this work with a final remark. Denote by $c_0$ the logarithmic capacity and by  $c_\alpha$ the $\alpha$-capacity for  $0<\alpha<1$. The case of Dirichlet spaces $\mathcal{D}_{\alpha}^{2}$, $0\leq \alpha<1$,  was studied in \cite{24,25}. In particular, it was shown in that 
if $f\in \mathcal{D}_{\alpha}^{2}\cap \mathcal{A}(\D)$,  is an outer function such that $\mathcal{Z}(f) $  is a generalized cantor set,  then $f$ is cyclic in $ \mathcal{D}_{\alpha}^{2}$ 
if and only if $c_\alpha(\mathcal{Z}(f))=0 $. We do not know if this result also holds for $ \mathcal{D}_{\alpha}^{p}\cap \mathcal{A}(\D)$.\\
\end{rema} 
{\bf Acknowledgments. }The authors are grateful to Sasha Borichev  for helpful discussions and for pertinent remarks.

\end{document}